\documentclass[12pt]{amsart}

\usepackage{amscd,amssymb,amsthm,verbatim}

\newcommand{\C}{{\mathbb C}}

\newcommand{\dvol}{\operatorname{dvol}}
\newcommand{\End}{\operatorname{End}}

\newcommand{\GL}{\operatorname{GL}}

\newcommand{\Hess}{\operatorname{Hess}}

\newcommand{\inj}{\operatorname{inj}}

\newcommand{\Lip}{\operatorname{Lip}}

\newcommand{\R}{{\mathbb R}}

\newcommand{\Ric}{\operatorname{Ric}}

\newcommand{\supp}{\operatorname{supp}}

\newcommand{\Z}{{\mathbb Z}}

\numberwithin{equation}{section}
\theoremstyle{plain}
\newtheorem{assumption}[equation]{Assumption}

\newtheorem{proposition}[equation]{Proposition}

\theoremstyle{definition}
\newtheorem{definition}[equation]{Definition}

\theoremstyle{definition}
\newtheorem{example}[equation]{Example}

\theoremstyle{definition}
\newtheorem{remark}[equation]{Remark}
\def\<{\langle}
\def\>{\rangle}
\def\({\left(}
\def\){\right)}

\begin{document}

\title{Ricci measure for some singular Riemannian metrics}
\author{John Lott}
\address{Department of Mathematics\\
University of California - Berkeley\\
Berkeley, CA  94720-3840\\ 
USA} \email{lott@berkeley.edu}

\thanks{Research partially supported
by NSF grant DMS-1207654}
\date{August 15, 2015}
\subjclass[2000]{}

\begin{abstract}
We define the Ricci curvature, as a measure, 
for certain singular torsion-free connections on the
tangent bundle of a manifold.
The definition uses an integral formula and vector-valued half-densities.
We give relevant examples in which the Ricci measure can be computed.
In the time dependent setting, we give a weak notion of a Ricci flow
solution on a manifold.
\end{abstract}

\maketitle

\section{Introduction} \label{sect1}

There has been much recent work about metric measure spaces with
lower Ricci bounds, particularly the Ricci limit spaces that arise 
as measured Gromov-Hausdorff limits of smooth manifolds with
a uniform lower Ricci bound. In this paper we address the question
of whether one can make sense of the Ricci curvature itself on singular
spaces.

From one's intuition about a two dimensional cone with total cone angle
less than $2 \pi$, the
Ricci curvature should exist at best as a measure.
One natural approach toward a weak notion of Ricci curvature is to
use an integral formula, such as the Bochner formula.
The Bochner identity says that if $\omega_1$ and $\omega_2$ are smooth
compactly supported
$1$-forms on a smooth Riemannian manifold $M$ then
\begin{equation} \label{1.1}
\langle \omega_1, \Ric(\omega_2) \rangle =
\int_M \left( \langle d \omega_1, d \omega_2 \rangle +
\langle d^* \omega_1, d^* \omega_2 \rangle -
\langle \nabla \omega_1, \nabla \omega_2 \rangle \right) \: \dvol.
\end{equation}
Equivalently,
\begin{equation} \label{1.2}
\langle \omega_1, \Ric(\omega_2) \rangle =
\int_M \sum_{i,j} \left( \nabla^i \omega_{1,i} \nabla^j \omega_{2,j}
- \nabla^i \omega_{1,j} \nabla^j \omega_{2,i} \right) \: \dvol.
\end{equation}

Now consider a possibly singular Riemannian metric on $M$.
In order to make sense of (\ref{1.2}), one's first attempt may be to
require that $\nabla \omega_1$ and $\nabla \omega_2$ are square integrable.
However, in the case of a two dimensional cone with total cone angle
less than $2 \pi$, if one requires square integrability
then one does not find any contribution
from the vertex of the cone.  That is, one would conclude that the
cone is Ricci flat, which seems wrong.
In order to see the curvature at the vertex, one
needs to allow for more general test forms. It is not 
immediately evident what
precise class of test forms should be allowed. A related fact is that
the expression for the Ricci tensor, using local coordinates, makes 
distributional sense if the Christoffel symbols are square integrable.  
However, this is not the case for the cone.

Our resolution to this problem is by first passing from 
$1$-forms to vector fields, and then passing to
vector-valued half-densities. For $V$ and $W$ vector-valued half-densities, 
we consider the quadratic form
\begin{equation} \label{1.3}
Q(V,W) = \int_M \sum_{i,j} \left[ \left( \nabla_i V^i \right) 
\left( \nabla_j W^j \right) -
\left( \nabla_i V^j \right) 
\left( \nabla_j W^i \right) \right].
\end{equation}
A compactly supported density on a
manifold can be integrated,
so two compactly supported half-densities can be multiplied and integrated. 
We require that $V$ and $W$
are compactly supported and Lipschitz regular on $M$.  
If one rewrote $Q$ using $1$-forms as in
(\ref{1.2}) then this would prescribe that the $1$-forms 
should lie in certain weighted spaces.  

One sees that (\ref{1.3})
does not involve the Riemannian metric directly, but
can be written entirely in terms of the connection.  Hence we work in 
the generality of torsion-free connections on the tangent bundle.
We also work with $C^{1,1}$-manifolds $M$, with an eye toward limit spaces;
it is known that Ricci limit spaces have a weak $C^{1,1}$-structure
\cite{Honda (2014)}. Then we say that a possibly singular connection is
{\em tame} if (\ref{1.3}) makes sense for all compactly supported 
Lipschitz vector-valued half-densities $V$ and $W$. We characterize
tame connections in terms of integrability properties of 
their Christoffel symbols.  We show
that a tame connection, with $Q$ bounded below, 
has a Ricci curvature that is well-defined
as a measure with values in $S^2(T^*M) \otimes {\mathcal D}^*$,
where ${\mathcal D}$ is the density line bundle.

We prove stability results for this Ricci measure.
We give examples to illustrate its meaning.

\begin{proposition} \label{1.4}
The Levi-Civita connection is tame and has a 
computable Ricci measure in the following cases :
Alexandrov surfaces,
Riemannian manifolds with boundary that are glued together, and
families of cones.
\end{proposition}

Passing to Ricci flow, 
one can use optimal transport to characterize supersolutions to the
Ricci flow equation on a manifold
\cite{Lott (2009),McCann-Topping (2010),Topping (2009)}.
There are also comparison principles for Ricci flow supersolutions
\cite{Bamler-Brendle (2014)},\cite[Section 2]{Ilmanen-Knopf (2003)}.
We give a weak notion of
a Ricci flow solution (as opposed to supersolution), in the sense that
the curvature tensor is not invoked, 
again on a fixed $C^{1,1}$-manifold.  
The idea is that the Ricci tensor appearing in
the Bochner integral formula can cancel the Ricci tensor appearing on the
right-hand side of the Ricci flow equation.  One could try  to formulate
such a time dependent integral identity just using the Bochner
equality for $1$-forms. However,
one would get a term coming from the time derivative of the volume form, which
unfortunately involves the scalar curvature.  Using
vector-valued half-densities instead,
this term does not appear.  We give examples of weak Ricci flow solutions,
along with a convergence result and a compactness result.

To mention some earlier work, Lebedeva and Petrunin
indicated the existence of a measure-valued curvature operator on an
Alexandrov space that is a noncollapsed limit of Riemannian manifolds
with a lower sectional curvature bound
\cite{Lebedeva-Petrunin (2008)}. The Ricci form
exists as a current on certain normal K\"ahler spaces and was 
used by Eyssidieux, Guedj and Zeriahi for K\"ahler-Einstein metrics 
\cite{Eyssidieux-Guedj-Zeriahi (2009)}. Naber gave a notion of
bounded Ricci curvature, in particular Ricci flatness,
for metric measure spaces \cite{Naber (2013)}.
Gigli discussed Ricci curvature for
certain metric measure spaces \cite[Section 3.6]{Gigli (2014)}.

The structure of this paper is as follows.  In Section \ref{sect2} we give
some background information.  Section \ref{sect3} has the definitions of 
tame connection and Ricci measure, and proves some properties of these.
Section \ref{sect4} gives some relevant examples. Section \ref{sect5} is
about weak Ricci flow solutions.

I thank the referee for helpful comments.

\section{Background} \label{sect2}

Let $M$ be an $n$-dimensional smooth manifold.
Let $FM$ denote the principal $\GL(n, \R)$-frame bundle of $M$.
For $c \in \R$, let $\rho_c :
\GL(n, \R) \rightarrow \GL(1, \R)$ be the homomorphism given by
$\rho_c(M) = |\det M|^{-c}$. There is an associated real
line bundle ${\mathcal D}^c = FM \times_{\rho_c} \R$, the 
$c$-density bundle. There is an isomorphism
${\mathcal D}^c \otimes {\mathcal D}^{c^\prime} \rightarrow
{\mathcal D}^{c+c^\prime}$ of line bundles. 
A section of ${\mathcal D}^c$ is called a $c$-density on $M$. 
A $1$-density is just called a density.  Compactly supported
densities on $M$ can be integrated, to give a linear functional
$\int_M : C_c(M; {\mathcal D}) \rightarrow \R$. There is a canonical
inner product on compactly supported half-densities,
given by $\langle f_1, f_2 \rangle = \int_M f_1 f_2$.

Let $\nabla$ be a torsion-free
connection on $TM$. There is an induced connection on ${\mathcal D}$.
Given a compactly supported vector-valued density $V$, i.e. a section 
$V \in C_c^\infty(M; TM \otimes
{\mathcal D})$, the integral $\int_M \sum_i \nabla_i V^i$ of its divergence
vanishes. With this fact, one can justify integration by parts.

The curvature of $\nabla$ is a section of
$\End(TM) \otimes \Lambda^2(T^*M)  = TM \otimes T^*M \otimes 
(T^*M \wedge T^*M)$. In terms of the latter description, 
the Ricci curvature of $\nabla$ is the covariant $2$-tensor field on $M$
obtained by contracting the $TM$ factor with the
first $T^*M$ factor in $(T^*M \wedge T^*M)$. In terms of indices,
$R_{jl} = \sum_i R^i_{\: \: jil}$. 

In this generality, the Ricci curvature
need not be symmetric. As $\nabla$ is torsion-free, the first
Bianchi identity holds and one finds that
$R_{jl} - R_{lj} = \sum_i R^i_{\: \: ijl}$. That is, the antisymmetric part of
the Ricci tensor is the negative of the curvature of the induced connection on
${\mathcal D}$, and represents an obstruction to the
local existence of a nonzero parallel density. Of course, if $\nabla$ is the
Levi-Civita connection of a Riemannian metric then there is a nonzero
parallel density, namely the Riemannian density.

Now let 
$M$ be a $C^{1,1}$-manifold.  This means that there is an
atlas $M = \bigcup_{\alpha} U_\alpha$ whose
transition maps $\phi_{\alpha \beta}$ have a first derivative
that is Lipschitz. We can take a maximal such atlas.
The preceding discussion of $c$-densities still makes sense in
this generality.

A $C^{1,1}$-manifold admits an underlying smooth structure,
in that we can find a subatlas with smooth transition
maps. Furthermore, any two such smooth structures are diffeomorphic.

\section{Tame connections and Ricci measure} \label{sect3}

In this section we give the notion of a tame connection and define
its Ricci measure.  In Subsection \ref{subsect3.1} we define tame
connections, characterize them in terms of the Christoffel symbols,
and prove stability under $L^\infty$-perturbations
of the connection. Subsection \ref{subsect3.2} has the definition of the
Ricci measure.  
In Subsection \ref{BE} we extend the notion of Ricci measure to the case of
weighted manifolds.
Subsection \ref{subsect3.3} is about 
singular Riemannian metrics
and Killing fields.

\subsection{Tame connections} \label{subsect3.1}

Let $M$ be an $n$-dimensional $C^{1,1}$-manifold.  
It makes sense to talk about the space
${\mathcal V}_{Lip}(M)$ of Lipschitz vector fields on $M$,
meaning Lipschitz-regular sections of $TM$.  Similarly, it makes
sense to talk about
Lipschitz vector-valued half-densities, i.e.
Lipschitz-regular sections of $TM \otimes {\mathcal D}^{\frac12}$.

Let ${\mathcal V}_{meas}(M)$ denote the measurable vector fields on $M$.
Let $\nabla$ be a measurable torsion-free connection on $TM$, i.e. an
$\R$-bilinear
map $\nabla : {\mathcal V}_{Lip}(M) \times {\mathcal V}_{Lip}(M) 
\rightarrow {\mathcal V}_{meas}(M)$ such that for 
$f \in \Lip(M)$ and $X,Y \in {\mathcal V}_{Lip}(M)$, we have
\begin{itemize}
\item
$\nabla_{fX} Y = f \nabla_X Y$,
\item
$\nabla_X (fY) = (Xf) Y + f \nabla_X Y$,
\item 
$\nabla_X Y - \nabla_Y X = [X,Y]$. 
\end{itemize}
Writing $\nabla_{\partial_i} \partial_j = \sum_k \Gamma^k_{\:\: ji}
\partial_k$, the Christoffel symbols $\Gamma^k_{\:\: ij}$ are measurable.

If $V$ is a vector-valued half-density then we can locally write it
as $V = \sum_j V^j \partial_j$, where $V^i$ is a locally defined
half-density. Then
$\nabla_i V = \sum_j (\nabla_i V^j) \partial_j$, where
$\nabla_i V^j$ is also a half-density.
Further writing $V^j = v^j \sqrt{dx^1 \ldots dx^n}$,
we have 
\begin{equation} \label{3.1}
\nabla_i V = \sum_j (\nabla_i v^j) \partial_j \otimes \sqrt{dx^1 \ldots dx^n},
\end{equation}
where
\begin{equation} \label{3.2}
\nabla_i v^j = \partial_i v^j + \sum_k \Gamma^j_{\: \: ki} v^k 
- \frac12 \sum_k \Gamma^k_{\: \: ki} v^j
\end{equation}

Given compactly supported 
Lipschitz vector-valued half-densities $V$ and $W$ on $M$, consider
the formal expression
\begin{equation} \label{3.3}
Q(V,W) = \int_M \sum_{i,j} \left[ \left( \nabla_i V^i \right) 
\left( \nabla_j W^j \right) -
\left( \nabla_i V^j \right) 
\left( \nabla_j W^i \right) \right]. 
\end{equation}
Note that the integrand of (\ref{3.3}) is a density on $M$.

\begin{definition} \label{3.4}
The connection $\nabla$ is {\em tame} if the integrand in (\ref{3.3}) is
integrable for all $V$ and $W$.
\end{definition}

If $n = 1$ then $Q$ vanishes identically.

\begin{remark} \label{3.5}
Suppose that $\nabla$ is the Levi-Civita connection of a Riemannian metric $g$.
We can use the Riemannian half-density 
$\sqrt{\dvol} = (\det g)^{\frac14} \sqrt{dx^1 \ldots dx^n}$
to trivialize ${\mathcal D}^{\frac12}$. 
Using this trivialization, 
there is an isometric isomorphism between 
vector-valued half-densities and $1$-forms, under which a
$1$-form $\omega = \sum_i \omega_i dx^i$ corresponds to a vector-valued
half-density $V = \sum_i v^i \partial_i \otimes \sqrt{dx^1 \ldots dx^n}$
with $v^i = \sum_j g^{ij} (\det g)^{\frac14} \omega_j$. In this case,
$Q$ could be computed using (\ref{1.2}), with the restriction on 
$\omega$ that each $v^i$ in the local description of 
its isomorphic vector-valued
half-density $V$ should be Lipschitz.
\end{remark}

We now characterize tameness of a connection in terms of its
Christoffel symbols.

\begin{proposition} \label{3.6}
Suppose that $n > 1$.
The connection $\nabla$ is tame if and only if in any 
coordinate neighborhood,
each $\Gamma^i_{\: \: jk}$ is locally integrable and each
$\sum_{i,j} (\Gamma^i_{\: \: kl} \Gamma^j_{\: \: ji} -
\Gamma^j_{\: \: ki} \Gamma^i_{\: \: lj})$ is locally integrable.
\end{proposition}
\begin{proof}
Suppose that $\nabla$ is tame. Let $U$ be a
coordinate neighborhood and choose $m \in U$. 
The point $m$ has a neighborhood $S$ with compact
closure in $U$.
Take $V$ and $W$ to have compact support in $U$. One finds that
\begin{equation} \label{3.7}
Q(V,W) = Q_1(V,W)+ Q_2(V,W),
\end{equation}
where
\begin{align} \label{3.8}
Q_1(V,W) = \int_U \sum_{i,k,l} & 
\left[ \frac12 (\partial_k v^k) \Gamma^i_{\: \: il} w^l
 + \frac12
v^k \Gamma^i_{\: \: il} (\partial_k w^l)
+ \frac12 (\partial_l v^k) \Gamma^i_{\: \: ik} w^l
+ \frac12 v^k \Gamma^i_{\: \: ik} (\partial_l w^l) \right. \\
& \left. - (\partial_i v^k) \Gamma^i_{\: \: kl} w^l -
v^k \Gamma^i_{\: \: kl} (\partial_i w^l) \right] dx^1 \ldots dx^n \notag
\end{align}
and
\begin{equation} \label{3.9}
Q_2(V,W) = \int_U \sum_{i,j,k,l} v^k \left( \Gamma^i_{\: \: kl}
\Gamma^j_{\: \: ji} -
\Gamma^j_{\: \: ki} \Gamma^i_{ \: \: lj} \right) w^l dx^1 \ldots dx^n.
\end{equation}

Given 
constant vectors $\{c^k\}_{k=1}^n$ and $\{d^l\}_{l=1}^n$, we can choose
$V$ and $W$ so that $v^k = c^k$ and $w^l = d^l$ in $S$.
Then the integrand of $Q_1$ vanishes in $S$. Hence
the integrability of the integrand of $Q(V, W)$ implies 
the integrability of 
$\sum_{i,j,k,l} c^k \left( \Gamma^i_{\: \: kl}
\Gamma^j_{\: \: ji} -
\Gamma^j_{\: \: ki} \Gamma^i_{ \: \: lj} \right) d^l$ in $S$,
for any choice of $\{c^k\}_{k=1}^n$ and $\{d^l\}_{l=1}^n$.
Letting $m$ and $S$ vary, this is equivalent to the local integrability of 
$\sum_{i,j} \left( \Gamma^i_{\: \: kl}
\Gamma^j_{\: \: ji} -
\Gamma^j_{\: \: ki} \Gamma^i_{ \: \: lj} \right)$ in $U$, for all
$k$ and $l$.

Returning to general $V$ and $W$ with 
compact support in $U$, we now know that the
integrand of $Q_2$ is integrable.  Hence the integrability of the
integrand of $Q$ implies the integrability of the integrand of
$Q_1$. 
Given a
constant matrix $\{C_r^{\: \: k} \}_{r,k=1}^n$ and a constant vector
$\{d^l \}_{l=1}^n$,
we can choose
$V$ and $W$ so that $v^k = \sum_r C_r^{\: \: k} x^r$ and 
$w^l = d^l$ in $S$. Then the integrand of $Q_1$, over $S$, becomes
\begin{equation} \label{3.10}
\sum_{i,k,l}
\left[ \frac12 C_k^{\: \: k} \Gamma^i_{\: \: il} d^l
+ \frac12 C_l^{\: \: k} \Gamma^i_{\: \: ik} d^l
- C_i^{\: \: k} \Gamma^i_{\: \: kl} d^l \right].
\end{equation}
Taking first $C_k^{\: \: l} = \delta_k^{\: \: l}$, we see
that $\frac{n-1}{2} \sum_i \Gamma^i_{\: \: il} d^l$ is integrable in
$S$ for
any choice of $\{d^l\}_{l=1}^n$. Hence $\sum_i \Gamma^i_{\: \: il}$
is integrable in $S$ for all $l$. It now follows from (\ref{3.10}) that 
$\sum_{i,k,l} C_i^{\: \: k} \Gamma^i_{\: \: kl} d^l$ is integrable
in $S$ for any choices of
$\{C_r^{\: \: k} \}_{r,k=1}^n$ and 
$\{d^l \}_{l=1}^n$.  Hence $\Gamma^i_{\: \: kl}$ is integrable in $S$
for all $i$, $k$ and $l$, so $\Gamma^i_{\: \: kl}$ is locally
integrable in $U$.

For the other direction of the proposition, suppose that in any 
coordinate neighborhood,
each $\Gamma^i_{\: \: jk}$ is locally integrable and each
$\sum_{i,j} (\Gamma^i_{\: \: kl} \Gamma^j_{\: \: ji} -
\Gamma^j_{\: \: ki} \Gamma^i_{\: \: lj})$ is locally integrable.
Given $V$ and $W$ with compact support, we can cover
$\supp(V) \cup \supp(W)$ by a finite number
$\{U_r\}_{r=1}^N$ of open sets, each
with compact closure in a coordinate neighborhood. 
Let $\{\phi_r\}_{r=1}^N$ be a subordinate
Lipschitz partition of unity.  Then
\begin{equation} \label{3.11}
Q(V,W) = \sum_{r=1}^N Q(\phi_r V, W).
\end{equation}
Looking at (\ref{3.8}) and (\ref{3.9}), we see that the integrand of
$Q(\phi_r V, W)$ has support in $U_r$. Then from (\ref{3.8}) and (\ref{3.9}),
we see that the integrand of $Q(\phi_r V, W)$ is integrable.
The proposition follows. 
\end{proof}

\begin{proposition} \label{3.12}
Suppose that we have a fixed collection of 
coordinate neighborhoods that cover $M$, in each of which
$\Gamma^i_{\: \: jk}$ is locally integrable and 
$\sum_{i,j} (\Gamma^i_{\: \: kl} \Gamma^j_{\: \: ji} -
\Gamma^j_{\: \: ki} \Gamma^i_{\: \: lj})$ is locally integrable.
Then 
in any other  
coordinate neighborhood, 
$\Gamma^i_{\: \: jk}$ is locally integrable and 
$\sum_{i,j} (\Gamma^i_{\: \: kl} \Gamma^j_{\: \: ji} -
\Gamma^j_{\: \: ki} \Gamma^i_{\: \: lj})$ is locally integrable.
\end{proposition}
\begin{proof}
This follows indirectly from the proof of Proposition \ref{3.6}, but
can also be seen directly from the transformation formula
\begin{equation} \label{3.13}
\widetilde{\Gamma}^i_{\: \: jk} = 
\sum_{a,b,c} \frac{\partial y^i}{\partial x^a}
\frac{\partial x^b}{\partial y^j}
\frac{\partial x^c}{\partial y^k}
\Gamma^a_{\: \: bc} +
\sum_a \frac{\partial y^i}{\partial x^a} \cdot
\frac{\partial^2 x^a}{\partial y^j \partial y^k}
\end{equation}
for the Christoffel symbols under a change of coordinate
from $x$ to $y$, along with the fact that
$\frac{\partial y}{\partial x}$ and $\frac{\partial x}{\partial y}$ 
are Lipschitz, and
$\frac{\partial^2x^a}{\partial y^j \partial y^k}$ is $L^\infty$.
\end{proof}

We now show that tameness is preserved by bounded perturbations of the
connection.

\begin{proposition} \label{3.14}
Suppose that $\nabla$ is tame.  Suppose that $T \cong T^i_{\: \: jk}$ is
a measurable 
$(1,2)$-tensor field, symmetric in the lower two indices,
such that for all Lipschitz vector fields $v$,
the $(1,1)$-tensor field $T(v) = \sum_k T^i_{\: \: jk} v^k$ is a locally
bounded section of $\End(TM)$. Then $\nabla + T$ is tame.
\end{proposition}
\begin{proof}
Writing $T_i = T(\partial_i)$, we have
\begin{align} \label{3.15}
& \int_M \sum_{i,j} \left[ \left( \nabla_i V^i + T_i V^i \right) 
\left( \nabla_j W^j + T_j W^j \right) -
\left( \nabla_i V^j + T_i V^j \right) 
\left( \nabla_j W^i + T_j W^i \right) \right] = \\
& \int_M \sum_{i,j} \left[ \left( \nabla_i V^i \right) 
\left( \nabla_j W^j  \right) -
\left( \nabla_i V^j  \right) 
\left( \nabla_j W^i \right) \right] + \notag \\
& \int_M \sum_{i,j} \left[ \left( T_i V^i \right) 
\left( \nabla_j W^j \right)
+
\left( \nabla_i V^i  \right) 
\left( T_j W^j \right)
 -
\left( T_i V^j \right) 
\left( \nabla_j W^i \right)
-
\left( \nabla_i V^j \right) 
\left( T_j W^i \right)
 \right] + \notag \\
& \int_M \sum_{i,j} \left[ \left( T_i V^i \right) 
\left( T_j W^j \right) -
\left( T_i V^j \right) 
\left( T_j W^i \right) \right]. \notag
\end{align}
As before, we take $V$ and $W$ to have compact support.
Since Proposition \ref{3.6} tells us that in each coordinate neighborhood, the
Christoffel symbols $\Gamma^i_{\: \: jk}$ are locally $L^1$, it
follows that $\nabla V$ and $\nabla W$ are $L^1$ on $M$. Since
$TV$ and $TW$ are $L^\infty$, the integrands in the 
second and third integrals on the
right-hand side of (\ref{3.15}) are integrable.  The proposition follows.
\end{proof}

\subsection{Ricci measure} \label{subsect3.2}

Suppose that $\nabla$ is tame.  We can rewrite the expression
for $Q_1(V, W)$ in (\ref{3.8}) as
\begin{equation} \label{3.16}
Q_1(V,W) = \int_U \sum_{i,k,l} 
\left[ \frac12 \Gamma^i_{\: \: il} \partial_k (v^k  w^l)
+ \frac12 \Gamma^i_{\: \: ik} \partial_l (v^k  w^l)
- \Gamma^i_{\: \: kl} \partial_i (v^k  w^l)
 \right] dx^1 \ldots dx^n.
\end{equation}
Using an underlying smooth structure for $M$, and taking
$V$ and $W$ to be smooth with compact support in
the coordinate neighborhood $U$ for the moment, 
it follows that
\begin{equation} \label{3.17}
Q(V, W) = \int_U \sum_{k,l} v^k R_{(kl)} w^l dx^1 \ldots dx^n,
\end{equation}
where the distribution 
\begin{equation} \label{3.18}
R_{(kl)} = \sum_i \left( \partial_i \Gamma^i_{\: \: kl} - \frac12 \partial_k
\Gamma^i_{\: \: il} - \frac12 \partial_l \Gamma^i_{\: \: ik} \right) +
\sum_{i,j} \left( \Gamma^i_{\: \: kl}
\Gamma^j_{\: \: ji} -
\Gamma^j_{\: \: ki} \Gamma^i_{ \: \: lj} \right)
\end{equation}
is recognized as the symmetrized Ricci tensor.

Given a continuous vector bundle $E$ on $M$, let
${\mathcal M}(M; E)$ denote the dual space to the
topological vector space of compactly supported
continuous sections of $E^*$. We can think of an element of 
${\mathcal M}(M; E)$ as an $E$-valued measure on $M$.

In the rest of this subsection, we make the following assumption.

\begin{assumption} \label{ass}
There is some nonnegative $h \in {\mathcal M}(M; S^2(T^*M) \otimes
{\mathcal D}^*)$
so that for all $V$ and $W$,
\begin{equation}
Q(V,W) \ge - \int_M \langle V, h W \rangle.
\end{equation}
\end{assumption}

Assumption \ref{ass} implies that the distributional tensor field
$R_{(kl)} + h_{kl}$ is nonnegative.  It follows that it is a tensor-valued
measure, and hence so is $R_{(kl)}$.  
The conclusion is that there is some
${\mathcal R} \in {\mathcal M}(M; S^2(T^*M) \otimes {\mathcal D}^*)$ so that
for all 
compactly supported
Lipschitz vector-valued half-densities $V$ and $W$ on $M$, we
have 
\begin{equation} \label{3.19}
Q(V,W) = \int_M \langle V, {\mathcal R} W \rangle.
\end{equation}
We call ${\mathcal R}$ the Ricci measure of the connection $\nabla$.

We now prove a convergence result for the Ricci measure.

\begin{proposition} \label{3.20}
Let $\nabla$ be a tame connection with Ricci measure ${\mathcal R}$. 
Let $\left\{ T^{(r)} \right\}_{r=1}^\infty$
be a sequence of measurable $(1,2)$-tensor fields as in Proposition
\ref{3.14}.  
Suppose that the connections $\left\{ \nabla + T^{(r)} \right\}_{r=1}^\infty$
satisfy Assumption \ref{ass} with a uniform choice of $h$.
Let $\left\{ {\mathcal R}^{(r)} \right\}_{r=1}^\infty$ 
be the Ricci measures of the connections 
$\left\{ \nabla + T^{(r)} \right\}_{r=1}^\infty$.
Suppose that for each compactly supported Lipschitz vector field $v$,
we have $\lim_{r \rightarrow \infty} T^{(r)}(v) = 0$ in
$L^\infty(M; \End(TM))$. Then $\lim_{r \rightarrow \infty}
{\mathcal R}^{(r)} = {\mathcal R}$ in the weak-$*$ topology on
${\mathcal M}(M; S^2(T^*M) \otimes {\mathcal D}^*)$.
\end{proposition}
\begin{proof}
Let $Q^{(r)}$ be the quadratic form associated to the tame connection
$\nabla + T^{(r)}$. From (\ref{3.15}), for any $V$ and $W$, we have
$\lim_{r \rightarrow \infty} Q^{(r)}(V,W) = Q(V,W)$. It follows that
$\lim_{r \rightarrow \infty} {\mathcal R}^{(r)} = {\mathcal R}$ 
distributionally.  Since the relevant distributions are all measures,
with ${\mathcal R}^{(r)} + h$ nonnegative,
we have weak-$*$ convergence.
\end{proof}

\begin{remark} \label{lq}
If we further assume that $\nabla$ has Christoffel symbols in $L^q_{loc}$, for
$q > 1$, then we reach the same conclusion under the weaker
assumption that
$\lim_{r \rightarrow \infty} T^{(r)}(v) = 0$ in
$L^{\max(2,q^*)}(M; \End(TM))$, where $\frac{1}{q} + \frac{1}{q^*} = 1$.
\end{remark}

\subsection{Bakry-Emery-Ricci measure} \label{BE}

Let $\nabla$ be a measurable torsion-free connection on $TM$.
We say that $f \in W^{1,1}_{loc}(M)$ if 
in any coordinate neighborhood
$U$, we have $f \in L^1_{loc}(U)$ and there are $S_i \in 
L^1_{loc}(U)$ so that for any Lipschitz functions $F^i$ with
compact support in $U$, we have
\begin{equation}
\int_U f \sum_i \partial_i F^i \: dx^1 \ldots dx^n = - 
\int_U \sum_i S_i F^i \: dx^1 \ldots dx^n.
\end{equation}
We let $\nabla_i f$ denote $S_i$. Given compactly supported Lipschitz
vector-valued half-densities $V$ and $W$ on $M$, consider the formal
expression
\begin{align} \label{added}
Q_f(V,W) = & \int_M \sum_{i,j} \left[ e^f  \left( \nabla_i 
\left( e^{- \: \frac{f}{2}} V^i \right) \right) 
\left( \nabla_j \left( e^{- \: \frac{f}{2}} W^j \right) \right) - \right. \\
& \left.  
\: \: \: \: \: \: \: \:
\: \: \: \: \: \: \: \:
e^{-f} \left( \nabla_i \left( e^{\frac{f}{2}} V^j  \right) \right) 
\left( \nabla_j \left( e^{\frac{f}{2}} W^i \right) \right) \right] \notag \\
= & \: \int_M \sum_{i,j} \left[
\left( \nabla_i V^i \right) 
\left( \nabla_j W^j \right) -
\left( \nabla_i V^j \right) 
\left( \nabla_j W^i \right) - \right. \notag \\
& \left.
\: \: \: \: \: \: \: \:
\: \: \: \: \: \: \: \:
\frac12 (\nabla_i f) 
\left( \nabla_j \left( V^i W^j
+ V^j W^i \right) \right) \right].
\notag
\end{align}
We say that the pair $(\nabla, f)$ is tame if the integrand in
(\ref{added}) is integrable for all $V$ and $W$. 

\begin{example}
Suppose that $f$ is semiconvex.
Then $\Hess(f) = \nabla \nabla f$ is well-defined as a measurable
symmetric $2$-tensor.
Suppose that $\nabla$ is tame in the sense of Definition \ref{3.4}
and $Q$ satisfies Assumption \ref{ass}.  
Then $(\nabla, f)$ is tame and
\begin{equation}
Q_f(V,W) = \int_M \langle V, {\mathcal R}_f W \rangle,
\end{equation} 
where the $S^2(T^*M) \otimes {\mathcal D}^*$-valued measure ${\mathcal R}_f$
is given in a coordinate neighborhood $U$ by
\begin{equation}
Q_f(V, W) = \int_U \sum_{k,l} v^k \left( R_{(kl)} + \Hess(f)_{kl} \right) w^l
\: dx^1 \ldots dx^n
\end{equation}
for $V$ and $W$ having compact support in $U$.
In the Riemannian setting, 
we recognize $R_{(kl)} + \Hess(f)_{kl}$ as the Bakry-Emery-Ricci tensor.
\end{example}

\subsection{Riemannian metrics} \label{subsect3.3}

Let $g$ be a Riemannian metric on $M$,
i.e. a measurable section of $S^2(T^*M)$ that is positive definite
almost everywhere.

\begin{definition} \label{w11}
A Riemannian metric $g$ lies in $W^{1,1}_{loc}$ if in
any coordinate neighborhood $U$, we have
$g_{ij} \in L^1_{loc}(U)$ and there are 
$S^l_{\: \: jk} \in L^1_{loc}(U)$ 
so that for any Lipschitz functions $\{f^{ijk}\}$ 
with compact support in $U$, we have
\begin{equation}
\int_U \sum_{i,j,k} g_{ij} \: \partial_k f^{ijk} \: dx^1 \ldots dx^n =
\int_U \sum_{i,j,k,l} g_{il} S^l_{\: \: jk} f^{ijk} \: dx^1 \ldots dx^n.
\end{equation}
\end{definition}

If $g \in W^{1,1}_{loc}$ then it has a Levi-Civita connection with
Christoffel symbols in $L^1_{loc}$.

We recall the classical result that a smooth compact Riemannian manifold
with negative Ricci curvature has no nonzero Killing vector fields.
(In fact, this is the only result that we know for manifolds with negative
Ricci curvature in dimension greater than two.)  We show that
there is an analogous result in our setting.

\begin{definition}
A Lipschitz vector-valued half-density $V$ is {\em Killing} if
\begin{equation}
\sum_{k} \left( g_{jk} \nabla_i V^k + g_{ik} \nabla_j V^k \right) = 0.
\end{equation}
\end{definition}

We note that if $V$ is a Killing vector-valued half-density then
writing $V = v \otimes \sqrt{\dvol_g}$, the vector field $v$
is a Killing vector field in the
usual sense, at least where $g$ is $C^1$.

\begin{proposition}
If $M$ is compact, the Levi-Civita connection is tame and 
$Q(V, V) < 0$ for all nonzero $V$, then there is no nonzero
Killing $V$.
\end{proposition}
\begin{proof}
If $V$ is Killing then $\sum_i \nabla_i V^i = 0$ and
$\nabla_j V^i = - \sum_{k,l} g_{jl} g^{ik} \nabla_k V^l$, so
\begin{equation}
Q(V,V) = \int_M \sum_{i,j,k,l} g_{jl} g^{ik} (\nabla_i V^j) (\nabla_k V^l)
\ge 0.
\end{equation}
The proposition follows.
\end{proof}

\section{Examples} \label{sect4}

In this section we compute examples of the Ricci measure
coming from Riemannian metrics in $W^{1,1}_{loc}$, in the sense of
Definition \ref{w11}.  The examples are Alexandrov
surfaces, Riemannian manifolds with boundary that are glued together,
families of cones, K\"ahler manifolds and limit spaces of manifolds
with lower bounds on Ricci curvature and injectivity radius.
At the end of the section we make some remarks.

\subsection{Alexandrov surfaces} \label{subsect4.1}

We recall that there is a notion of a metric on a surface having
bounded integral curvature \cite{Reshetnyak (1993)}. This includes
surfaces with Alexandrov curvature bounded below. 
(For us, the relevance of the latter is that they are the noncollapsed
Gromov-Hausdorff limits of smooth Riemannian two-manifolds
with Ricci curvature bounded below.)
Such a metric
comes from a (possibly) singular Riemannian metric $g$. There exist local
isothermal coordinates in which $g = e^{2 \phi} \left(
(dx^1)^2 + (dx^2)^2 \right)$,
where $\phi$ is the difference of two subharmonic functions 
(with respect to the Euclidean metric). The volume density $\dvol_g$,
given locally by 
$e^{2 \phi} dx^1 dx^2$, lies in $L^1_{loc}$. 

\begin{proposition} \label{4.1}
The Levi-Civita connection $\nabla$ is tame.
\end{proposition}
\begin{proof}
In the isothermal coordinates, we have 
\begin{equation} \label{4.2}
\Gamma^i_{\: \: jk} = \delta_{ij} \partial_k \phi + \delta_{i,k} 
\partial_j \phi
- \delta_{jk} \partial_i \phi.
\end{equation}
A subharmonic function $f$ on a two-dimensional domain has
$\partial_i f \in L^1_{loc}$ \cite[Pf. of Lemma 1.6]{Landkof (1972)}.
(The proof there is for functions defined on $\C$ but can be localized.)  
One finds that
$\sum_{i,j} (\Gamma^i_{\: \: kl} \Gamma^j_{\: \: ji} -
\Gamma^j_{\: \: ki} \Gamma^i_{\: \: lj}) = 0$.
Proposition \ref{3.6} implies that $\nabla$ is tame.
\end{proof}

The quadratic form $Q$ can be described in a coordinate-free way
as follows.  Let
$K$ denote the curvature measure of $g$, given in local isothermal
coordinates by
$dK = - (\partial_1^2 + \partial_2^2) \phi \: dx^1 dx^2$.
(That is, for any smooth function $f$ with support in the coordinate
chart, $\int_M f \: dK = - \int_{\R^2} 
(\partial_1^2 + \partial_2^2)f \:  \phi \: dx^1 dx^2$.)
Given compactly supported Lipschitz vector-valued half-densities $V$ and $W$, 
consider $\frac{g(V,W)}{\dvol_g}$. We claim that this extends over
the singularities of $g$ to a
Lipschitz function on $M$. To see this, 
in isothermal coordinates we can write
$V = \sum_i v^i \partial_i \otimes \sqrt{dx^1 dx^2}$ and
$W = \sum_i w^i \partial_i \otimes \sqrt{dx^1 dx^2}$, with
$\{v^i\}_{i=1}^2$ and $\{w^i\}_{i=1}^2$ Lipschitz. Then
$g(V, W) = e^{2 \phi} \left( v^1 w^1 + v^2 w^2 \right) dx^1 dx^2$ and
$\dvol_g = e^{2 \phi} dx^1 dx^2$, from which the claim follows.
One finds
\begin{proposition} \label{4.3}
\begin{equation} \label{4.4}
Q(V, W) = \int_M \frac{g(V,W)}{\dvol_g} \: dK.
\end{equation}
\end{proposition}

\begin{example}
If $g$ is smooth and $\kappa$ is the Gaussian curvature then
$dK = \kappa \dvol_g$, so 
\begin{equation}
Q(V, W) = \int_M \kappa \: g(V, W).
\end{equation}
\end{example}

\begin{example} \label{4.5}
For $\alpha < 1$, put $g = 
\left( (x^1)^2 + (x^2)^2 \right)^{- \alpha}
\left( (dx^1)^2 + (dx^2)^2 \right)$. Then
$(\R^2, g)$ is a cone with total cone angle $2 \pi (1 - \alpha)$.
One finds that
\begin{equation} \label{4.6}
Q(V, W) = 2 \pi \alpha \left( v^1(0,0) w^1(0,0) + v^2(0,0) 
w^2(0,0) \right).
\end{equation}
In this case, 
$\partial_i \phi = - \alpha \frac{x^i}{(x^1)^2 + (x^2)^2}$,
so $\Gamma^i_{\: \: jk}$ lies in $L^1_{loc}$ but not in $L^2_{loc}$.
\end{example}

\begin{example} \label{4.7}
For $c > 0$, suppose that
$g = e^{- 2c |x^1|} \left( (dx^1)^2 + (dx^2)^2 \right)$. Then
\begin{equation} \label{4.8}
Q(V, W) = 2c \int_{- \infty}^\infty \left( v^1(0, x^2) w^1(0, x^2)  + 
v^2(0, x^2) w^2(0, x^2) \right) dx^2.
\end{equation}

In this example, $\Gamma^i_{\: \: jk}$ lies in $L^2_{loc}$. 
The geometry can be described as follows.  Take a
two-dimensional cone with total cone angle $2 \pi c$. Truncate the cone
at distance $\frac{1}{c}$ from the vertex.  Take two copies of such
truncated cones and glue them along their circle boundaries.  Remove
the two vertex points and take the universal cover. 

Note that
all of the tangent cones are isometric to $\R^2$, but the Ricci measure
is not absolutely continuous.
\end{example}

\subsection{Gluing manifolds along their boundaries} \label{subsect4.2}

Let $M_1$ and $M_2$ be Riemannian manifolds with boundaries. 
Let $A^{(1)}_{ij}$ (resp. $A^{(2)}_{ij}$) denote the second fundamental form
of $\partial M_1$ (resp. $\partial M_2$), taking values in normal vectors.
Let $H^{(1)}$ (resp. $H^{(2)}$)
denote the mean curvature of
$\partial M_1$ (resp. $\partial M_2$), also taking value in normal vectors.  
Our conventions are such that for the
unit ball in $\R^n$, if $v$ is a nonzero tangent vector to the unit sphere
then $A(v,v)$ is inward pointing.

Let
$\phi : \partial M_1 \rightarrow \partial M_2$ be an isometric
diffeomorphism.
Using the
local product structure near $\partial M_1$ (resp. $\partial M_2$)
coming from the normal exponential map, the result
$M = M_1 \cup_{\phi} M_2$ of gluing $M_1$ to $M_2$ acquires a smooth
structure. It also acquires a $C^0$-Riemannian metric.
Let $X \subset M$ denote the gluing locus. 

Given a compactly supported Lipschitz vector-valued half-density $V$
on $M$, using the isomorphism
$TM \big|_X = TX \oplus N_XM$, we can decompose $V$ on $X$ as
$V = V^T + V^\perp$, where $V^T$ is a section of $TX \otimes
{\mathcal D}_M^{\frac12} \big|_X$ and
$V^\perp$ is a section of $N_XM \otimes
{\mathcal D}_M^{\frac12} \big|_X$.  

Given $x \in X$, let $n_x$ be the inward pointing unit normal vector 
to $M_1$ at $x$. Given $V$ and $W$, decompose them along $X$ as
$V = V^T + V^\perp$ and $W = W^T + W^\perp$. Then
$\langle A^{(1)}(V^T_x, W^T_x) - A^{(2)}(V^T_x, W^T_x), n_x \rangle$ 
lies in ${\mathcal D}_M \big|_{\{x\}}$. 
We would get the same result if we switched
the roles of $M_1$ and $M_2$.

Similarly, $\langle V^\perp_x, W^\perp_x \rangle$ lies in
${\mathcal D}_M \big|_{\{x\}}$. We can compute the number $\langle
H^{(1)}_x - H^{(2)}_x, n_x \rangle$.
We would get the same result if we switched
the roles of $M_1$ and $M_2$.

Let $x^0$ be a local coordinate at $x$ so that 
$\frac{\partial}{\partial x^0}$ is a unit
normal to $X$ at $x$. There is a unique
linear map ${\mathcal T}_x : {\mathcal D}_M \big|_{\{x\}}
\rightarrow {\mathcal D}_X \big|_{\{x\}}$ so that
$dx^0 \otimes {\mathcal T}_x(\omega_x) = \omega_x$. This extends to a map
${\mathcal T} : C^\infty \left( X; {\mathcal D}_M \big|_X \right)
\rightarrow C^\infty(X; {\mathcal D}_X)$.

\begin{proposition} \label{4.9}
\begin{align} \label{4.10}
Q(V, W)  = & \int_{M_1} \langle V, \Ric(W) \rangle + 
\int_{M_2} \langle V, \Ric(W) \rangle \: + \\
& \int_X {\mathcal T} \left( 
\langle A^{(1)}(V^T, W^T) - A^{(2)}(V^T, W^T), n \rangle
+ \langle H^{(1)} - H^{(2)}, n \rangle \langle V^\perp, W^\perp \rangle
\right). \notag
\end{align}
\end{proposition}
\begin{proof}
The Levi-Civita connection $\nabla$ on $M$ has $\{ \Gamma^i_{\: \: jk} \}$
in $L^2_{loc}$, 
so we can just compute the usual Ricci tensor (\ref{3.18}) as a
distribution.
On the interior of $M_1$ (resp. $M_2$), we clearly get the usual Ricci
tensor of $M_1$ (resp. $M_2$), so it suffices to look at what happens
near $X$.
Since $V$ and $W$ are compactly supported, we can
effectively reduce to the case when $X$ is compact.
We can choose a local coordinate $x^0$ near $X$, with
$n = \frac{\partial}{\partial x^0} \big|_X$ pointing into $M_1$, so that
the metric takes the form 
\begin{equation} \label{4.11}
g = (dx^0)^2 + h(x_0) + O \left( (x^0)^2 \right).
\end{equation}
Here we have a metric $h(x^0)$ on $X$ for $x^0 \in (-\epsilon, \epsilon)$.
As a function of $x^0$, the metric $h$ is continuous on $(- \epsilon,
\epsilon)$, smooth on $[0, \epsilon)$ and smooth on $(- \epsilon, 0]$.
The second fundamental form of $\partial M_1$ (resp. $\partial M_2$) is
$A^{(1)} = - \frac12 \left( \lim_{x^0 \rightarrow 0^+} \frac{dh}{dx^0}
\right) n$ 
(resp. $A^{(2)} = 
- \frac12 \left( 
\lim_{x^0 \rightarrow 0^- }\frac{dh}{dx^0} \big|_{x^0 = 0} \right) n$).
Using local coordinates $\{x^i\}$ on $X$, we have
\begin{align} \label{4.12}
\lim_{x^0 \rightarrow 0^+} \Gamma^0_{\: \: ij} & = \langle A^{(1)}_{ij}, 
n \rangle, \\
\lim_{x^0 \rightarrow 0^-} \Gamma^0_{\: \: ij} & = \langle A^{(2)}_{ij}, n 
\rangle, \notag \\
\lim_{x^0 \rightarrow 0^+} \Gamma^i_{\: \: j0} & = - \langle A^{i,(1)}_{\:
\: j}, n \rangle, \notag \\
\lim_{x^0 \rightarrow 0^-} \Gamma^i_{\: \: j0} & = - \langle 
A^{i,(2)}_{\: \: j}, n \rangle, \notag
\end{align}
The relevant terms in (\ref{3.18}) are
\begin{align} \label{4.13}
R_{(kl)} & = \partial_0 \Gamma^0_{\: \: kl}+ \ldots, \\
R_{(00)} & = - \partial_0 \Gamma^i_{\: \: i0} + \ldots. \notag
\end{align}
Hence the singular part of the Ricci measure is
\begin{align} \label{4.14}
R_{(kl),sing} & = \langle A^{(1)}_{kl} - A^{(2)}_{kl}, 
n \rangle \delta_0(x^0), \\
R_{(00),sing} & = \langle H^{(1)} - H^{(2)}, 
n \rangle \delta_0(x^0). \notag
\end{align}
The proposition follows.
\end{proof}

\begin{example} \label{4.15}
Let $M_1$ and $M_2$ each be the result of taking a two-dimensional
cone with total cone angle $2 \pi c$ and trunctating it at a distance
$L$ from the vertex. Then the contribution to the Ricci measure of $M$
from the circle gluing locus is
$\frac{2}{L} \int_X {\mathcal T} \langle V, W \rangle$. This is consistent
with Example \ref{4.7}. 
\end{example}

\begin{remark}
Based on Proposition \ref{4.9}, if $M$ is a Riemannian manifold with boundary
then it would be natural to consider
\begin{equation}
Q(V, W)  =  \int_{M} \langle V, \Ric(W) \rangle + 
\int_{\partial M} {\mathcal T} \left( 
\langle A(V^T, W^T), n \rangle
+ \langle H, n \rangle \langle V^\perp, W^\perp \rangle
\right)
\end{equation}
to define the Ricci measure of $M$.
\end{remark}

\subsection{Families of cones} \label{subsect4.3}

We first consider the case of a single cone.

\begin{proposition} \label{4.16}
For $\alpha < 1$, put $g = |x|^{-2\alpha} \sum_{i=1}^n (dx^i)^2$
on $\R^n$. Then the Levi-Civita connection is tame.
If $n = 2$ then $Q(V, W)$ is given by (\ref{4.6}). If $n > 2$ then
$Q(V, W) = \int_{\R^n} \langle V,\Ric(W) \rangle$. That is, if $n > 2$
then there is no singular contribution to the Ricci measure from 
the vertex of the cone.
\end{proposition}
\begin{proof}
The case $n=2$ was handled in Example \ref{4.5}.  If $n > 2$ then we can use
the formula for conformal transformations from (\ref{4.2}), with
$\phi = - \alpha \ln |x|$. In this case $\partial_i \phi \in L^2_{loc}$, so the
formula (\ref{3.18}) makes sense as a distribution.  However, since
$|\partial_i \phi| \sim |x|^{-1}$, there is no
contribution to $Q(V,W)$ from the origin.
(To have such a contribution, one would need
to have $|\partial_i \phi| \sim |x|^{-(n-1)}$.)
\end{proof}

We now consider a family of cones.
Let $\pi : M \rightarrow B$ be an $n$-dimensional 
real vector bundle over a Riemannian manifold $B$. 
Given $b \in B$, we write $M_b = \pi^{-1}(b)$.
Let $h$ be a Euclidean
inner product on $M$ and let $D$ be an $h$-compatible connection.
There is a natural Riemannian metric $g_0$ on $M$ with $\pi : M \rightarrow B$
being a Riemannian submersion, so that the restrictions of $g_0$ to fibers
are specified by $h$, and with horizontal subspaces coming from 
$D$.
Let $s : B \rightarrow M$ denote the zero section and let
$Z$ denote its image. Given $\alpha < 1$, 
let $g$ be the Riemannian metric on
$M - Z$ obtained from $g_0$, at $m \in M-Z$, by multiplying the fiberwise 
component of $g_0$ by $h(m,m)^{- \alpha}$.

Given $z \in Z$, let
$\dvol_{M_{\pi(z)},z}$ denote the Riemannian density at $z$ of the fiber 
$M_{\pi(z)}$, induced from $h_z$.
There is a
unique linear map ${\mathcal T}_z : {\mathcal D}_{M} \big|_{\{z\}} 
\rightarrow {\mathcal D}_{Z} \big|_{\{z\}}$ so that 
$\dvol_{M_{\pi(z)},z} \otimes {\mathcal T}_z(\omega_z) = \omega_z$ for all
$\omega_z \in {\mathcal D}_{M} \big|_{\{z\}}$. This extends to a linear map
${\mathcal T} : C^\infty(Z; {\mathcal D}_M \big|_Z) \rightarrow
C^\infty(Z; {\mathcal D}_Z)$.

Given a compactly supported Lipschitz vector-valued half-density $V$ on $M$,
we can decompose its restriction to $Z$ orthogonally (with respect to $g_0$)
as $V \big|_Z = V^{tan} + V^{nor}$, where $V^{tan}$ is tangential to $Z$ and
$V^{nor}$ is normal to $Z$, i.e. tangential to the fibers of the vector
bundle.

\begin{proposition} \label{4.17}
The Levi-Civita connection $\nabla$ of $g$ is tame.
If $n=2$ then
\begin{equation} \label{4.18}
Q(V, W) = \int_M \langle V, \Ric(W) \rangle +
2 \pi \alpha \int_Z  {\mathcal T}
\left( \langle V^{nor}, W^{nor} \rangle_{g_0} \right).
\end{equation}
If $n > 2$ then $Q(V, W) = \int_M \langle V, \Ric(W) \rangle$.
\end{proposition}
\begin{proof}
We can choose local coordinates $\{x^\beta, x^i\}$ for $M$ so that
the coordinates
$\{x^\beta\}$ pullback from $B$ and 
the coordinates $\{x^i\}$  restrict to the fibers as
linear orthogonal coordinates with respect to $h$. In terms of such
coordinates, we can write 
\begin{equation} \label{4.19}
g_0 = \sum_{\beta, \gamma} k_{\beta \gamma} dx^\beta dx^\gamma + 
\sum_i \left( dx^i + \sum_{\beta, j} 
C^i_{\: \: j \beta} x^j dx^\beta \right)^2,
\end{equation} 
where $\{C^i_{\: \: j \beta}\}$ is the local description of the
connection $D$ and $\{k_{\beta \gamma} \}$ is the local description of
the Riemannian metric on $B$. Then
\begin{equation} \label{4.20}
g = \sum_{\beta, \gamma} k_{\beta \gamma} dx^\beta dx^\gamma + 
\left( \sum_l (x^l)^2 \right)^{- \alpha} \sum_i \left( dx^i + \sum_{\beta, j} 
C^i_{\: \: j \beta} x^j dx^\beta \right)^2,
\end{equation} 

Let $\{\widehat{\Gamma}^\beta_{\: \: \gamma \delta}\}$ denote the
Christoffel symbols of the Riemannian metric $k$ on $B$.
Put
\begin{equation} \label{4.21}
F^i_{\: \: j \beta \gamma} = 
\partial_\beta C^i_{\: \: j \gamma} - 
\partial_\gamma C^i_{\: \: j \beta} + 
\sum_k C^i_{\: \: k \beta} C^k_{\: \: j \gamma} -
\sum_k C^i_{\: \: k \gamma} C^k_{\: \: j \beta},
\end{equation}
the curvature of $D$.

Given $b \in B$, we can choose the coordinates $\{x^i\}$ near
the fiber $M_b$ so that $C^i_{\: \: j \beta}(b) = 0$. Then on 
$M_b$, we have 
\begin{equation} \label{4.22}
g \big|_{M_b} = 
\sum_{\beta, \gamma} k_{\beta \gamma} dx^\beta dx^\gamma + 
\left( \sum_l (x^l)^2 \right)^{- \alpha} \sum_i \left( dx^i \right)^2.
\end{equation}
One finds that on $M_b$,
\begin{align} \label{4.23}
\Gamma^i_{\: \: jk} & = - \frac{\alpha}{|x|^2} \left( x_k \delta^i_j +
x^j \delta^i_k - x^i \delta_{jk} \right), \\
\Gamma^i_{\: \: \beta \gamma} & = \frac12 \sum_j x^j \left( 
\partial_\gamma C^i_{\: \: j \beta} + \partial_\beta C^i_{\: \: j \gamma}
\right),
\notag \\
\Gamma^{\beta}_{\: \: \gamma i} & =
\Gamma^{\beta}_{\: \: i \gamma} = - \frac12 |x|^{- 2 \alpha}
\sum_{\sigma, j} k^{\beta \sigma}  x^j F^i_{\: \: j \sigma \gamma}, \notag \\
\Gamma^{\beta}_{\: \: \gamma \delta} & = 
\widehat{\Gamma}_{\: \: \gamma \delta}, \notag \\
\Gamma^{\beta}_{\: \: ij} & = \Gamma^{i}_{\: \: j \beta} =
\Gamma^{i}_{\: \: \beta j} = 0. \notag 
\end{align}
Using Proposition \ref{3.6}, one can check that $\nabla$ is tame.

The Ricci curvature of $\nabla$ can be computed using the
splitting of $TM$ into its vertical and horizontal components relative to
$\pi$.  The corresponding O'Neill formulas still hold for the
Ricci measure. In the present case, the fibers
of $\pi$ are totally geodesic with respect to the metric $g$. Relative to the
vertical orthonormal coframe 
\begin{equation} \label{4.24}
\tau^i = |x|^{- \alpha}
\left( dx^i + \sum_j C^i_{\: \: j \beta} x^j 
dx^\beta \right)
\end{equation}
and a local orthonormal coframe $\{\tau^\beta\}$ for $k$,
one finds that the curvature of the horizontal distribution is
given by 
\begin{equation} \label{4.25}
A^i_{\: \: \beta \gamma} = \frac12 |x|^{- \alpha} \sum_j
F^i_{\: \: j \beta \gamma} x^j.
\end{equation}
Then using the O'Neill formulas, as given in 
\cite[(4.7)]{Lott (2014)}, one finds that the only singular contribution to the
Ricci measure is the fiberwise contribution coming from the
singular points. Using Proposition \ref{4.16}, the proposition follows. 
\end{proof}

\subsection{K\"ahler manifolds} \label{subsect4.4}

Let $M$ be a complex manifold of complex dimension $n$.
Suppose that $M$ admits
a K\"ahler metric $h$ which is $W^{1,1}_{loc}$-regular, 
in the sense of Definition
\ref{w11}.
Suppose that the Levi-Civita connection is tame.  The Ricci
measure of $(M, h)$ can be described as follows. 
Let $V$ be a compactly supported Lipschitz section of
$T^{(1,0)}M \otimes {\mathcal D}^{\frac12}$ and let 
$W$ be a compactly supported Lipschitz section of
$T^{(0,1)}M \otimes {\mathcal D}^{\frac12}$.
Then $Q(V, W) = \int_M q(V, W)$, where the measure
$q(V, W)$ has the following description in local coordinates.
Write $V = \sum_i v^i \partial_{z^i} \otimes 
\sqrt{dx^1 dy^1 \ldots dx^n dy^n}$ and 
$W = \sum_j w^{\overline{j}} \partial_{\overline{z}^j} \otimes 
\sqrt{dx^1 dy^1 \ldots dx^n dy^n}$. Then
\begin{equation} \label{4.26}
q(V,W) = - \sum_{i,j} \left( \partial_{z^i} \partial_{\overline{z}^j}
\log \det h \right) \: v^i \: w^{\overline{j}} \: dx^1 dy^1 \ldots dx^n dy^n.
\end{equation}

\subsection{Limit spaces of manifolds with lower bounds on 
Ricci curvature and injectivity radius} \label{andersoncheeger}

Given $n \in \Z^+$, $K \in \R$ and $i_0 > 0$, let
$(X, x)$ be a pointed Gromov-Hausdorff limit of a sequence
$\{(M_i, m_i, g_i)\}_{i=1}^\infty$
of complete $n$-dimensional pointed Riemannian manifolds 
with $\Ric(g_i) \ge K g_i$ and $\inj_{m_i} \ge i_0$. From
\cite[Theorem 0.2 and p. 268]{Anderson-Cheeger (1992)}, 
for any $p \in (n, \infty)$, the 
space $X$ is a $L^{2,p}$-manifold with a Riemannian metric $g_X$
that is locally $L^{1,p}$-regular. In particular, for any $\alpha \in (0,1)$
the manifold $X$ is also
a $C^{1,\alpha}$-manifold and hence has an underlying smooth structure,
which is unique up to diffeomorphism
\cite[Theorem 2.10]{Hirsch}.
In order to apply the formalism of this paper, we extend the smooth structure
to a $C^{1,1}$-structure. Any two such $C^{1,1}$-structures
are related by a homeomorphism $\phi$ of $X$ that is
$L^{2,p}$-regular for all $p \in (n, \infty)$.

As mentioned, with respect to the $L^{2,p}$-manifold structure, 
$g_X$ is locally $L^{1,p}$-regular (and also locally $C^\alpha$-regular).  
Hence the same will be true with
respect to the smooth structure, and the ensuing $C^{1,1}$-structure.
By H\"older's inequality, $g_X$ is also locally $L^{1,1}$-regular.
Since $g_X$ is nondegenerate and continuous in local coordinates, 
it has a continuous inverse.
Putting $S^l_{\: \: jk} = - \sum_i g^{li} \partial_k g_{ij}$,  
the metric $g_X$ lies in $W^{1,1}_{loc}$ in the sense of Definition 
\ref{w11}. 
Its Christoffel symbols lie in $L^2_{loc}$ (since $g_X$ is locally
$L^{1,2}$-regular). Hence the
Levi-Civita connection of $g_X$ is tame. 
Also from \cite{Anderson-Cheeger (1992)}, for large $i$ there
are pointed diffeomorphisms $\phi_i : (X, x) \rightarrow (M, m_i)$
so that $\lim_{i \rightarrow \infty} \phi_i^* g_i = g_X$ in
$L^{1,p}_{loc}$. By Remark \ref{lq}, it follows that
$Q(V,W) \ge K \int_X g_{X,ij} V^i W^j$. Hence there is a Ricci measure
${\mathcal R}$.

Since the Christoffel symbols lie in $L^2_{loc}$, the
quadratic form $Q$ can be extended to compactly supported $V$ and $W$ in
$\bigcap_{p \in (n, \infty)} L^{1,p}$, the latter of which also lies in 
$\bigcap_{\alpha \in (0,1)} C^\alpha$.
It follows that
$Q$ is covariant with respect to diffeomorphisms $\phi$ of $X$ that
are $L^{2,p}$-regular for all $p \in (n, \infty)$.
Hence ${\mathcal R}$ is independent of the choice of
$C^{1,1}$-structure.

\subsection{Remarks} \label{subsect4.5}

\begin{remark} \label{4.27}
Let $g$ be a Riemannian metric on $M^n$ which lies in $W^{1,1}_{loc}$, in the
sense of Definition \ref{w11}.
Suppose that the Levi-Civita connection is tame.

Suppose that the length metric gives a well-defined
compact metric space $X$.
If $Q \ge 0$ then a natural question is whether
$X$ has nonnegative $n$-Ricci curvature (with respect to the
Hausdorff measure) in the sense of
\cite{Lott-Villani (2009),Sturm (2006)}. 
One way to answer this would be to show
that $X$ is the Gromov-Hausdorff limit of a sequence of smoothings
$\{(M, g_i)\}_{i=1}^\infty$ of $(M,g)$, with the Ricci curvature of
$(M, g_i)$ bounded below by $- \frac{1}{i} g_i$. In the setting of
Subsection \ref{subsect4.2}, i.e. gluing Riemannian manifolds along boundaries,
the argument for this appears in
\cite[Section 4]{Perelman}.

Conversely, one can ask whether $X$ having nonnegative $n$-Ricci curvature
implies that $Q \ge 0$.
\end{remark}

\begin{remark} \label{4.28}
Suppose that the metric space $X$ of Remark \ref{4.27} has $Q \ge 0$ and 
nonnegative Ricci curvature in the sense of 
\cite{Lott-Villani (2009),Sturm (2006)}.
One can ask if there is a relationship between the possible
singularity of the Ricci measure and the existence of a singular
stratum of $X$ in the sense of \cite{Cheeger-Colding (1997)}.
Example \ref{4.7} shows that there is no direct relationship, since
the Ricci measure may be singular even if all of the tangent
cones are Euclidean.  However, one can ask whether the existence of
a codimension-two singular stratum (i.e. 
${\mathcal S}_{n-2} \neq {\mathcal S}_{n-3}$ 
in the notation of \cite{Cheeger-Colding (1997)})
forces the Ricci measure
to be singular.
\end{remark}

\begin{remark} \label{4.29}
If $X$ has no singular strata of codimension less than three
(i.e. ${\mathcal S} = {\mathcal S}_{n-3}$ 
in the notation of \cite{Cheeger-Colding (1997)})
then one can ask whether
a compactly supported 
Lipschitz vector-valued half-density $V$ necessarily has $\nabla V$
square integrable. The
relevance of this would be for nonmanifold spaces, where
the square integrability of $\nabla V$ on the regular set would be
a natural condition, whereas the requirement of $V$ being Lipschitz
may not make sense in a neighborhood of a singular point.
See Example \ref{5.15}. 

We note that the singular K\"ahler-Einstein
metrics considered in \cite{Eyssidieux-Guedj-Zeriahi (2009)} 
do not have any singular strata of real codimension two.
\end{remark}

\begin{remark}
Suppose that $X$ is a pointed Gromov-Hausdorff limit of a sequence of
complete $n$-dimensional pointed 
Riemannian manifolds with a uniform lower bound on
their Ricci curvature.  For concreteness, we consider the case when the
Hausdorff dimension of $X$ is $n$, i.e. when the sequence is noncollapsing.
As in \cite[Section 3]{Cheeger-Colding (2000)}, 
the complement of a measure zero subset of $X$ can
be covered by a countable union of sets, each of which is biLipschitz
to a Borel subset of $\R^n$. Furthermore, the transition maps between
such sets can be taken to be $C^{1,1}$-regular in a natural weak sense
\cite[Theorem 1.5]{Honda (2014)}. Using this structure, one can define
a Levi-Civita connection on $X$ with measurable Christoffel symbols 
\cite[Section 3]{Honda (2014)}. 
The first question is whether the Levi-Civita connection is
necessarily
tame, in the sense of satisfying the equivalent condition of
Proposition \ref{3.6}. This is
true in the setting of Subsection \ref{andersoncheeger}.

If the Levi-Civita connection is tame then it should be possible to use this
to construct the
Ricci measure on the regular set of $X$ (compare with
Example \ref{4.7} and Subsection \ref{andersoncheeger}).
Based on Proposition \ref{4.17}, one would expect that any reasonable
notion of the Ricci measure on the singular set
should vanish on ${\mathcal S}_{n-3}$, and be given
on ${\mathcal S}_{n-2}$ by an analog of the last term in (\ref{4.18}).
\end{remark}

\section{Weak Ricci flow} \label{sect5}

In this section we give notions of weak Ricci flow solutions.  In Subsection
\ref{subsect5.1} we prove an integral identity for smooth Ricci flow solutions.
In Subsection \ref{subsect5.2} we define tame Ricci flow solutions, give a 
compactness result and discuss some
examples.  In Subsection \ref{subsect5.3} we define the broader class of
cone-preserving Ricci flow solutions and give further examples.

\subsection{An integral identity} \label{subsect5.1}

Let $M$ be a smooth manifold.  Let $\{g(t)\}_{t \in [0, T)}$ 
be a smooth one-parameter family of Riemannian metrics on $M$.

\begin{proposition} \label{5.1}
Let $\{V(t)\}_{t \in [0,T)}$ and $\{W(t)\}_{t \in [0,T)}$
be one-parameter families of
vector-valued half-densities on $M$.
We assume that for each $T^\prime \in [0, T)$,
the family $V$ has compact support in $M \times [0, T^\prime]$ and
is Lipschitz there, and similarly for the family $W$.
Then
$(M, g(\cdot))$ satisfies the Ricci flow equation
\begin{equation} \label{rf}
\frac{dg}{dt} \: = \: - \: 2 \Ric_{g(t)}
\end{equation}
 if and only if for every
such $V$ and $W$, and every $t \in [0, T)$, we have
\begin{align} \label{5.2}
& \int_M \sum_{ij} g_{ij}(t) V^i(t) W^j(t) = 
\int_M \sum_{ij} g_{ij}(0) V^i(0) W^j(0) \: +  \\
& \int_0^t \int_M \sum_{i,j} \left[
g_{ij} (\partial_s V^i) W^j +  g_{ij} V^i (\partial_s W^j) -
2 \left( \nabla_i V^i  \right) \left( \nabla_j W^j \right) + 
2 \left( \nabla_i V^j \right) \left( \nabla_j W^i \right) \right](s) \:
ds. \notag
\end{align}
\end{proposition}
\begin{proof}
Suppose that $(M, g(\cdot))$ is a Ricci flow solution. Then
\begin{equation} \label{5.3}
\frac{d}{dt} \sum_{ij} g_{ij}(t) V^i(t) W^j(t) =  
- 2 \sum_{ij}  R_{ij} V^i W^j +
\sum_{ij} g_{ij} (\partial_t V^i) W^j + 
\sum_{ij} g_{ij} V^i (\partial_t W^j).
\end{equation}
Integrating gives (\ref{5.2}).  Conversely, if (\ref{5.2}) holds then
by taking $V$ and $W$ smooth and differentiating in $t$, we see that
(\ref{5.3}) holds for all smooth $V$ and $W$.  This implies that
(\ref{rf}) holds.
\end{proof}

\subsection{Tame Ricci flow} \label{subsect5.2}

Now let $M$ be a $C^{1,1}$-manifold.

\begin{definition} \label{5.4}
Let $\{g(t)\}_{t \in [0, T)}$ be a one-parameter family of
$W^{1,1}_{loc}$-Riemannian metrics on $M$, in the sense of
Definition \ref{w11}, which is locally-$L^1$ 
on $M \times [0, T)$.
Suppose that for each 
$t \in [0, T)$, the Levi-Civita connection of $g(t)$ is tame in
the sense of Definition \ref{3.4}. 
Suppose that there is an integrable function $c : (0, T) \rightarrow
\R^+$ so that for all $t \in (0, T)$, the time-$t$ Ricci measure satisfies
$Q(V, W) \ge -c(t) \int_M g_{ij}(t) V^i W^j$.
Let $\{V(t)\}_{t \in [0,T)}$ and $\{W(t)\}_{t \in [0,T)}$
be one-parameter families of
vector-valued half-densities on $M$.
We assume that for each $T^\prime \in [0, T)$,
the family $V$ has compact support in $M \times [0, T^\prime]$ and
is Lipschitz there, and similarly for the family $W$.
We say that
$\{g(t)\}_{t \in [0, T)}$ is a {\em tame Ricci flow solution} if (\ref{5.2}) is
satisfied for all such $V$ and $W$, and all $t \in [0, T)$.

\end{definition}

\begin{example} \label{5.5}
Let $\{h(t)\}_{t \in [0, \infty)}$ be a smooth Ricci flow solution on 
$M$.  Given a $C^{1,1}$-diffeomorphism $\phi$ of $M$, put
$g(t) = \phi^* h(t)$. Then $\{g(t)\}_{t \in [0, T)}$ is a tame Ricci 
flow solution. This is because 
equation (\ref{5.2}) for $g$, $V$ and $W$, is equivalent to 
equation (\ref{5.2}) for $h$, $\phi_* V$ and $\phi_* W$. Hence
Proposition \ref{5.1} applies.
\end{example}

\begin{example} \label{5.6}
For all $t \ge 0$, let $g(t)$ be the metric of Example \ref{4.5}.
We claim that if $\alpha \neq 0$ then $\{g(t)\}_{t \in [0, \infty)}$ is
not a tame Ricci flow solution.  This can be seen by taking 
$V$ and $W$ to be time-independent in (\ref{5.2}).
\end{example}

We show that the property of being a tame Ricci flow solution passes to
Lipschitz limits.

\begin{proposition} \label{5.7}
Let $M$ be a $C^{1,1}$-manifold.  Let 
$\{g_i(\cdot)\}_{i=1}^\infty$ be a sequence of
one-parameter families of Riemannian metrics on $M$,
each defined for $t \in [0, T)$ and locally-$L^1$ on $M \times [0, T)$, 
with each $g_i(t)$ locally Lipschitz, 
satisfying (\ref{5.2}). 
Let $\{g_\infty(t)\}_{t \in [0, T)}$ be a one-parameter family of
locally Lipschitz Riemanniann metrics on $M$.
Suppose that for all $T^\prime \in [0, T)$ and every coordinate
neighborhood $U \subset M$ with compact closure,
\begin{equation} \label{5.8}
\lim_{i \rightarrow \infty} \sup_{t \in [0, T^\prime]}
d_{\Lip(U)}(g_i(t), g_\infty(t)) = 0.
\end{equation}
Then $g_\infty(\cdot)$ is a tame Ricci flow solution.
\end{proposition} 
\begin{proof} 
From the convergence assumption, $g_\infty$ is locally-$L^1$ on
$M \times [0, T)$. The Christoffel symbols of 
$\{g_i(t)\}_{i=1}^\infty$ and $g_\infty(t)$ are all locally-$L^\infty$, so the
Levi-Civita connections are tame.  Because the Levi-Civita connections of
$\{g_i\}_{i=1}^\infty$ converge to that of 
$g_\infty$, in $L^\infty_{loc}$ on $M \times [0, T)$,
it follows that $g_\infty$ satisfies (\ref{5.2}). 
\end{proof}

\begin{example} \label{5.9}
Let $\{g(t)\}_{t \in [0, T)}$ be a smooth Ricci flow solution on a smooth
manifold $M$.
Let $\{\phi_i\}_{i=1}^\infty$ be a sequence of smooth diffeomorphisms of
$M$ that $C^{1,1}$-converge on compact subsets to a $C^{1,1}$-diffeomorphism
$\phi_\infty$ of $M$.  Then $\lim_{i \rightarrow \infty} \phi_i^* g(\cdot) =
\phi_\infty^* g(\cdot)$ with Lipschitz convergence on compact subsets of
$M \times [0, T)$, and $\phi_\infty^* g(\cdot)$ is a tame Ricci flow
solution.
\end{example}

We now give a compactness result for tame Ricci flow solutions.

\begin{proposition} \label{5.10}
Let $M$ be a $C^2$-manifold.  Let 
$\{g_i(\cdot)\}_{i=1}^\infty$ be a sequence of tame Ricci flow solutions,
defined for $t \in [0, T)$,
consisting of $C^1$-Riemannian metrics on $M$.
Suppose that 
for all $T^\prime \in [0, T)$ and every coordinate
neighborhood $U \subset M$ with compact closure, 
the $g_i(t)$'s are uniformly bounded above and below on 
$U \times [0, T^\prime]$, and the $g_i(t)$'s and their 
first spatial partial derivatives are 
uniformly bounded and
uniformly equicontinuous (in $i$) on $U \times [0, T^\prime]$.
Then after passing to a subsequence, there is a tame Ricci flow solution
$\{g_\infty(t)\}_{t \in [0, T)}$ on $M$ consisting of $C^1$-Riemannian metrics
so that the $g_i(t)$'s $C^1$-converge to $g_\infty(t)$, locally uniformly
on $M \times [0, T)$.
\end{proposition}
\begin{proof}
By a diagonal argument, after passing to a subsequence we can assume
that $\lim_{i \rightarrow \infty} g_i(\cdot) = g_\infty(\cdot)$ as stated
in the proposition.
Then Proposition \ref{5.7} shows that $g_\infty(\cdot)$ is a tame
Ricci flow solution.
\end{proof}

We now address when one can get a tame Ricci flow solution by appending
a time-zero slice to a smooth Ricci flow solution defined for positive time.

\begin{proposition} \label{5.11}
Let $\{g(t) \}_{t \in (0, T)}$ be a smooth Ricci flow solution
with $\Ric(g(t)) \ge - c(t) g(t)$ for some positive integrable function $c$.
Suppose that there is some $g(0) \in W^{1,1}_{loc}$ with tame
Levi-Civita connection so that
$\lim_{t \rightarrow 0^+} g(t) = g(0)$ in $L^1_{loc}$. Then
$\{g(t) \}_{t \in [0, T)}$ is a tame Ricci flow solution.
\end{proposition}
\begin{proof}
Let $V$ and $W$ be one-parameter families as in Definition \ref{5.4}.
For any $t^\prime \in [0, T)$ and $t \in [t^\prime , T)$, the 
analog of (\ref{5.2}) holds with $0$ replaced by $t^\prime$.
Taking $t^\prime \rightarrow 0$ shows that (\ref{5.2}) holds.
\end{proof}

\begin{example} \label{5.12}
Let $\Sigma$ be a compact two-dimensional metric space with
curvature bounded in the Alexandrov sense. 
One can construct a Ricci flow
solution starting from $\Sigma$, in an certain sense, which will be smooth for
positive time \cite{Richard (2012)}. Using
\cite[Lemma 3.3]{Richard (2012)} and Proposition \ref{5.11}, 
we can extend the solution
back to time zero to get a
tame Ricci flow solution $g(\cdot)$ that exists on some
time interval $[0, T)$, with $\Sigma$ corresponding to $g(0)$.
\end{example}

\begin{example} \label{5.13}
Let $(\R^2, g_0)$ be a 
two-dimensional metric cone with total cone angle in $(0, 2 \pi]$.
There is a corresponding expanding soliton 
$\{g(t)\}_{t  > 0}$ with the property that
at any positive time, the tangent cone at infinity is isometric to
$(\R^2, g_0)$
\cite[Section 2.4]{Chow-Knopf (2004)}.
Putting $g(0) = g_0$, one obtains a tame Ricci flow solution
$\{g(t)\}_{t \ge 0}$.
\end{example}

\begin{remark} \label{5.14}
Let $g_0$ be a Lipschitz-regular Riemannian metric on a 
compact $C^{1,1}$-manifold $M$. Choose a compatible smooth structure on $M$.
From \cite[Theorem 5.3]{Koch-Lamm (2013)}, there is a smooth solution
$\{h(t)\}_{t \in (0, T)}$ to the DeTurck-Ricci flow 
with $\lim_{t \rightarrow 0^+} h(t) = g_0$.
Fixing $t_0 \in (0, T)$ and using the estimate in
\cite[Theorem 5.3]{Koch-Lamm (2013)},
we can integrate the vector field in the DeTurck trick backward in time
from $t_0$ to $0$, to obtain a homeomorphism $\phi$ of $M$. We can
think of $\phi$ as giving a preferred smooth structure based on the
Riemannian metric $g_0$.
We can also undo the vector field in the DeTurck trick, starting at
time $t_0$, to obtain a smooth Ricci flow solution 
$\{\widehat{g}(t)\}_{t \in (0, T)}$. Presumably $\phi$ is 
$C^{1,1}$
and there is a tame Ricci flow solution starting from $g_0$, given by
$g(t) = \phi^* \widehat{g}(t)$.
\end{remark}

\subsection{Cone-preserving Ricci flow} \label{subsect5.3}

There has been recent work about Ricci flow on certain singular
spaces with conical singularities along codimension-two strata, 
under the requirement that the
flow preserve the conical singularities
\cite{Liu-Zhang (2014),Mazzeo-Rubinstein-Sesum 
(2013),Phong-Song-Sturm-Wang (2014),Shen 
(2014),Shen (2014b),Yin (2010),Yin (2013)}.
As seen in Example \ref{5.6}, such solutions may not be tame Ricci flow
solutions.  However, one can consider an alternative and less
restrictive definition,
which we call {\em cone-preserving Ricci flow solutions}, in which
the $V$ and $W$ of Definition \ref{5.4} are 
additionally required to have square-integrable
covariant derivative.

\begin{example} \label{5.15}
For all $t \ge 0$, let $g(t)$ be the metric of Example \ref{4.5}.
We claim
that $\{g(t)\}_{t \in [0, \infty)}$ is a cone-preserving Ricci flow
solution. To see this, we can assume that $\alpha \neq 0$.
Suppose that $V$ is a compactly supported
Lipschitz vector-valued half-density.  Writing
$V = \sum_i v^i \partial_i \otimes \sqrt{dx^1 dx^2}$, one finds
$\nabla_i V = \sum_j \nabla_i v^j \otimes \sqrt{dx^1 dx^2}$, where
\begin{align} \label{5.16}
\nabla_1 v^1 = & \partial_1 v^1 + (\partial_2 \phi) v^2, \\
\nabla_1 v^2 = & \partial_1 v^2 - (\partial_2 \phi) v^1, \notag \\
\nabla_2 v^1 = & \partial_2 v^1 - (\partial_1 \phi) v^2, \notag \\
\nabla_2 v^2 = & \partial_2 v^2 + (\partial_1 \phi) v^1. \notag
\end{align}
Here $\phi(x^1, x^2) = - \frac{\alpha}{2} \log \left( (x^1)^2 + 
(x^2)^2 \right)$. 
Now the square norm of $\nabla V$ is
\begin{equation} \label{5.17}
\int_M \sum_{i,j,k,l} 
g_{jl} g^{ik} (\nabla_i v^j) (\nabla_k v^l) \: dx^1 dx^2 =
\int_M \sum_{i,j} (\nabla_i v^j)^2 \: dx^1 dx^2.
\end{equation}
Suppose that this is finite. As each $v^i$ is Lipschitz, and
$\partial_i \phi = - \alpha \frac{x^i}{(x^1)^2 + 
(x^2)^2}$, it follows that $v^i(0) = 0$. Then 
(\ref{4.6}) gives that $Q(V, W) = 0$.
Looking at (\ref{5.2}), the claim follows.
\end{example}

This example shows that there may be
nonuniqueness among cone-preserving Ricci flow
solutions with a given initial condition, in view of the expanding
soliton solution mentioned in Example \ref{5.13}.

We see that in this example, the
cone angle along the codimension-two stratum is unchanged.
If one wants to give a notion of a weak Ricci flow solution 
along these lines on a nonmanifold
space (which we do not address here), 
it is probably natural to impose the square integrability of
$\nabla V$ as a requirement.

\begin{example} \label{5.18}
For $k,n > 1$, consider $\C^n/\Z_k$, where the generator of $\Z_k$ acts 
isometrically on the flat 
$\C^n$ as multiplication by $e^{2 \pi i/k}$. 
We expect that with any reasonable definition
of a weak Ricci flow, this will give a static Ricci flow solution
(since Proposition \ref{4.16} indicates that the vertex of a cone
with a smooth link, in real dimension greater than two, should not
contribute to the Ricci measure).

On the other hand, there is an expanding soliton solution that exists
for $t > 0$, and whose $t \rightarrow 0$ limit is $\C^n/\Z_k$
\cite[Section 5]{Feldman-Ilmanen-Knopf (2003)}. 
This shows that one cannot expect a uniqueness result for weak Ricci flow
solutions whose time-zero slice is a nonmanifold, without further
restrictions.
\end{example}

\end{document}